\documentclass{amsart}

\usepackage[T1]{fontenc}
\usepackage{lmodern}
\usepackage{microtype}

\usepackage{amsmath,amssymb,mathtools}
\usepackage{mathrsfs}

\usepackage{enumitem}

\usepackage{xcolor}
\usepackage[colorlinks=true,citecolor=teal,linkcolor=blue,urlcolor=blue]{hyperref}

\newif\ifdraft
\draftfalse

\ifdraft
  \usepackage{showkeys}
\fi

\makeatletter
\def\@maketitle{\normalfont\normalsize
  \let\@makefnmark\relax  \let\@thefnmark\relax
  \ifx\@empty\@date\else \@footnotetext{\@setdate}\fi
  \ifx\@empty\@subjclass\else \@footnotetext{\@setsubjclass}\fi
  \ifx\@empty\@keywords\else \@footnotetext{\@setkeywords}\fi
  \ifx\@empty\thankses\else \@footnotetext{\def\par{\let\par\@par}\@setthanks}\fi
  \global\topskip42\p@
  \@settitle
  \ifx\@empty\authors \else \@setauthors \fi
  \@setabstract
  \normalsize
  \dimen@34\p@ \advance\dimen@-\baselineskip
  \vskip\dimen@\relax
}
\def\@citestyle{\m@th\upshape\mdseries\color{teal}}
\def\@biblabel#1{\textcolor{teal}{\textbf{[#1]}}}
\makeatother

\numberwithin{equation}{section}

\theoremstyle{plain}

\newtheorem{theorem}{Theorem}[section]
\newtheorem{lemma}[theorem]{Lemma}
\newtheorem{proposition}[theorem]{Proposition}
\newtheorem{corollary}[theorem]{Corollary}

\theoremstyle{definition}

\newtheorem{definition}[theorem]{Definition}

\theoremstyle{remark}

\newtheorem{remark}[theorem]{Remark}

\everymath{\displaystyle}

\newcommand{\R}{\mathbb{R}}
\newcommand{\cR}{\mathcal{R}}
\newcommand{\cS}{\mathcal{S}}
\newcommand{\cT}{\mathcal{T}}

\newcommand{\Ric}{\operatorname{Ric}}
\newcommand{\Rm}{\operatorname{Rm}}
\newcommand{\divg}{\operatorname{div}}
\newcommand{\trg}{\operatorname{tr}}
\newcommand{\Id}{\operatorname{Id}}

\newcommand{\inner}[2]{\left\langle #1,#2\right\rangle}

\newcommand{\papertitle}{On Ricci solitons whose level hypersurfaces have parallel second fundamental form}
\newcommand{\papershorttitle}{Level hypersurfaces with parallel second fundamental form}

\title[\papershorttitle]{\papertitle}

\author[M. A. R. M. Hor\'acio]{Matheus Andrade Ribeiro de Moura Hor\'acio}

\address{
Departamento de Matem\'atica,
Universidade de Bras\'ilia,
70910-900, Bras\'ilia, DF,
Brazil
}

\email{matheus.andrade5488@gmail.com}

\date{}



\hypersetup{
  pdftitle={\papertitle},
  pdfauthor={Matheus Andrade Ribeiro de Moura Horacio},
  pdfsubject={Differential geometry},
  pdfkeywords={Gradient Ricci solitons, rigidity, constant scalar curvature, multiply warped products}
}

\begin{document}

\begin{abstract}
We study gradient Ricci solitons whose potential functions have level hypersurfaces with parallel second fundamental form. We show that such a soliton is locally a multiply warped product of a one-dimensional base and Einstein fibers, with potential function depending only on the base. Thus, locally, this condition characterizes the multiply warped structure that occurs in many classical constructions of Ricci solitons. We further prove that the same local structure follows for any gradient Ricci soliton if just one regular level hypersurface has parallel second fundamental form and parallel intrinsic Ricci tensor. As an application, we prove that a complete nonsteady gradient Ricci soliton with constant scalar curvature is rigid as soon as one regular level hypersurface has this property. In particular, the conjecture of Cao holds for gradient shrinking Ricci solitons with this property. We also prove that a complete gradient shrinking Ricci soliton whose regular level hypersurfaces have parallel second fundamental form and at most one nonzero principal curvature is rigid, without any assumption on the scalar curvature.
\end{abstract}

\maketitle

\section{Introduction and main results}\label{sec_intro}

A gradient Ricci soliton $(M,g,f,\lambda)$ is a Riemannian manifold $(M^n,g)$ together with a smooth function $f$ and a constant $\lambda\in\R$ satisfying
\begin{align}\label{fundeq}
	\Ric+\nabla^2f=\lambda g.
\end{align}
It is shrinking, steady or expanding according as $\lambda>0$, $\lambda=0$ or $\lambda<0$, and nonsteady if $\lambda\neq0$. Following Petersen and Wylie \cite{pewy}, it is called rigid if it is isometric to a quotient $N\times_\Gamma\R^k$ of a product of an Einstein manifold $N$, with Einstein constant $\lambda$, and a Euclidean space, where $f=\frac{\lambda}{2}|x|^2$ on the Euclidean factor and $\Gamma$ acts freely on $N$ and by orthogonal transformations on $\R^k$. By \cite[Theorem 1.2]{pewy}, a nonsteady gradient Ricci soliton is rigid if and only if it has constant scalar curvature and is radially flat, that is, $\Rm(\nabla f,X,\nabla f,X)=0$ for every $X$. Rigidity is also known under several other curvature conditions, see for instance \cite{flgr11,ms13,pw10}.

Cao conjectured that every complete gradient shrinking Ricci soliton with constant scalar curvature $R$ is rigid (see \cite{chengzhou}). The conjecture was proved in dimension four by Cheng and Zhou \cite{chengzhou}; see also \cite{ouquwu}. Wang and Wu \cite{wwres} proved that a complete gradient shrinking Ricci soliton with $R=(n-2)\lambda$ whose level sets of $f$ have vanishing Weyl tensor is a finite quotient of $\R^2\times\mathbb{S}^{n-2}$. It is worth mentioning that, when the scalar curvature is constant, the potential function is isoparametric, and Fern\'andez-L\'opez and Garc\'ia-R\'io used this fact in \cite{fr16} to show that $R\in\{0,\lambda,\ldots,n\lambda\}$. They also proved that such a soliton is rigid if and only if its Ricci operator has constant rank, or if and only if $\trg\Ric^3$ and $\trg\Ric^4$ are constant, and that it is rigid when its Ricci tensor has at most three distinct eigenvalues (see \cite[Theorem 2, Remark 17 and Theorem 9]{fr16}).

Many of the known gradient Ricci solitons are multiply warped products of a one-dimensional base and Einstein fibers, with $f$ depending only on the base. This is the case of the steady solitons of Bryant \cite{bryant} and Ivey \cite{ivey1}, of the steady and expanding solitons of Dancer and Wang \cite{danwan,danwan2}, which have an arbitrary number of Einstein fibers, and of the shrinking solitons in \cite{angenent}. Their potential functions are isoparametric, and complete gradient Ricci solitons with isoparametric potential functions were recently investigated by Tran and Yamazaki \cite{tranyamazaki}. These constructions start from a multiply warped ansatz, and it is natural to ask whether a counterexample to the conjecture of Cao can be found among such metrics. A converse question is which geometric conditions on the level hypersurfaces of $f$ force this ansatz.

In this article, we investigate gradient Ricci solitons whose level hypersurfaces of the potential function have parallel second fundamental form. As we will see, this condition characterizes, around regular points of $f$, the multiply warped products described above, and the conjecture of Cao holds for every complete soliton in this class. We write $\cR=\{x\in M\ :\ \nabla f(x)\neq0\}$ and $E_1=\nabla f/|\nabla f|$ on $\cR$. For $c\in\R$, we call $\Sigma=f^{-1}(c)\cap\cR$ a regular level hypersurface of $f$ when it is nonempty. We denote by $\cS X=\nabla_XE_1$ the shape operator of $\Sigma$, and by $\nabla^{\Sigma}$, $\Rm_{\Sigma}$, $\Ric_{\Sigma}$ and $R_{\Sigma}$ the Levi-Civita connection, the curvature tensor, the Ricci operator and the scalar curvature of the metric induced on $\Sigma$. We say that $\Sigma$ has parallel second fundamental form if $\nabla^{\Sigma}\cS=0$.

\begin{proposition}\label{prop_level}
Let $(M^n,g,f,\lambda)$ be a gradient Ricci soliton, and let $\Sigma$ be a regular level hypersurface of $f$. If $\Sigma$ has parallel second fundamental form, then $E_1$ is an eigenvector of $\Ric$ along $\Sigma$, and $|\nabla f|$, $R$ and the eigenvalues of $\Ric$ are constant on each connected component of $\Sigma$.
\end{proposition}

Since $\Delta f=n\lambda-R$, Proposition \ref{prop_level} shows that, if every regular level hypersurface has parallel second fundamental form, then $f$ is transnormal and isoparametric \cite{wang} in a neighborhood of each regular point. Our structural result is the following.

\begin{theorem}\label{thm_mwp}
Let $(M^n,g,f,\lambda)$ be a gradient Ricci soliton whose regular level hypersurfaces have parallel second fundamental form. Then every regular point of $f$ has a neighborhood isometric to a multiply warped product
\begin{align}\label{eq_mwp}
	I\times_{h_1}N_1^{r_1}\times\cdots\times_{h_k}N_k^{r_k},\qquad g=\mathrm{d}s^2+h_1^2g_{N_1}+\cdots+h_k^2g_{N_k},
\end{align}
with pairwise nonproportional warping functions, whose one-dimensional base $I$ is an integral curve of $E_1$, where $f$ depends only on $s$ and each fiber $N_i$ of dimension $r_i\geq1$ is Einstein. Furthermore, the regular level hypersurfaces of $f$ have parallel Ricci tensor.
\end{theorem}

The soliton equation on \eqref{eq_mwp}, with $f$ depending only on $s$, reduces to a system of ordinary differential equations for $h_1,\ldots,h_k$ and $f$, of the type studied by Ivey \cite{ivey1} and by Dancer and Wang \cite{danwan,danwan2} (see also \cite[Section 1]{dw11}). Together with Theorem \ref{thm_mwp}, this gives a local description of all gradient Ricci solitons whose regular level hypersurfaces have parallel second fundamental form. Conversely, the slices of \eqref{eq_mwp} are Riemannian products, as observed by Tran and Yamazaki for doubly warped products \cite[Remark 1.4]{tranyamazaki}, and their shape operator is $(h_i'/h_i)\,\Id$ on $TN_i$, so they have parallel second fundamental form. In this sense, Theorem \ref{thm_mwp} is a converse to the multiply warped ansatz underlying these constructions. The number $k$ of fibers is the number of distinct principal curvatures at a generic value of $s$, and the examples of Dancer and Wang show that it can be arbitrarily large. This should be compared with nontrivial gradient almost Ricci solitons with harmonic Weyl tensor, which, as shown in joint work with Borges and dos Santos \cite{bhs2026}, are locally multiply warped products with at most two fibers. Under the present hypothesis, no such bound occurs. On the other hand, the level hypersurfaces of the $U(m)$-invariant K\"ahler shrinking solitons of Feldman, Ilmanen and Knopf \cite{fik} are Berger spheres whose horizontal distribution, an eigendistribution of $\cS$, is not integrable, so these solitons of cohomogeneity one are not in our class.

The proof of Theorem \ref{thm_mwp} relies on an equation for the variation of $\cS$ along $E_1$, which is the tangential part of the soliton equation for a metric $\mathrm{d}s^2+g_s$ with $f=f(s)$ \cite[(1.16) and Remark 1.17]{dw11}, \cite[(2.6)]{tranyamazaki}.
The main point is that, when $\nabla^{\Sigma}\cS=0$, the induced metrics on the level hypersurfaces have the same Levi-Civita connection, and that, by this equation, the eigenspaces of $\cS$ on the flat factor of a local de Rham decomposition do not depend on $s$. This fixes the warped factors.

The hypothesis of Theorem \ref{thm_mwp} can be verified on a single level hypersurface.

\begin{theorem}\label{thm_one_level}
Let $(M^n,g,f,\lambda)$ be a gradient Ricci soliton. If some regular level hypersurface $\Sigma$ of $f$ satisfies
\[
	\nabla^{\Sigma}\cS=0\qquad\text{and}\qquad\nabla^{\Sigma}\Ric_{\Sigma}=0,
\]
then every regular level hypersurface of $f$ has parallel second fundamental form. Moreover, $\Ric(\nabla f,X)=0$ for every $X$ orthogonal to $\nabla f$.
\end{theorem}

By Theorem \ref{thm_mwp}, the regular level hypersurfaces have parallel Ricci tensor when they have parallel second fundamental form, so both conditions on $\Sigma$ are also necessary. The proof uses the analyticity of gradient Ricci solitons \cite{kotblms}. In geodesic coordinates around $\Sigma$, the soliton equation is an evolution system for the induced metrics, their shape operators and $f$, and the Cauchy data on $\Sigma$ coincide with those of a multiply warped product with Einstein fibers and potential function depending only on the base. Hence both solutions coincide near $\Sigma$, and analytic continuation gives the property on every regular level hypersurface.

We now turn to the conjecture of Cao. By \cite[Theorem 1.2]{pewy}, a nonsteady soliton with constant scalar curvature is rigid if and only if it is radially flat, and our approach is to derive radial flatness from the geometry of the level hypersurfaces and their collapse at the critical set of $f$.

\begin{definition}\label{def_tangential}
The Ricci tensor of a gradient Ricci soliton is tangentially parallel if $(\nabla_X\Ric)(Y,Z)=0$ for all $X,Y,Z$ orthogonal to $\nabla f$, at every regular point of $f$.
\end{definition}

This condition only involves the components of $\nabla\Ric$ tangent to the level hypersurfaces. Writing $\Ric_T$ for the restriction of $\Ric$ to $T\Sigma$, we have, for $X,Y,Z\in T\Sigma$,
\begin{align}\label{eq_tangential}
	(\nabla_X\Ric)(Y,Z)=(\nabla^{\Sigma}_X\Ric_T)(Y,Z)+\inner{\cS X}{Y}\Ric(E_1,Z)+\inner{\cS X}{Z}\Ric(Y,E_1),
\end{align}
since $\nabla_XY=\nabla^{\Sigma}_XY-\inner{\cS X}{Y}E_1$. Both derivatives coincide when $\Ric(E_1,\cdot)$ vanishes on $T\Sigma$, which is the case when $R$ is constant, since $\nabla R=2\Ric(\nabla f,\cdot)$ \cite{ccz,pewy}. The Ricci tensor of a multiply warped product \eqref{eq_mwp} satisfying \eqref{fundeq}, with $f$ depending only on $s$, is tangentially parallel. Indeed, $\nabla^2f=f''\mathrm{d}s^2+f'\nabla^2s$, where $\nabla^2s$ vanishes on $\partial_s$ and coincides with $\inner{\cS\cdot}{\cdot}$ on $T\Sigma$, so, by \eqref{fundeq}, $\Ric(E_1,\cdot)$ vanishes on $T\Sigma$ and $\Ric_T=\lambda\,\Id-f'\cS$, with $f'$ constant on each slice. Since the slices have parallel second fundamental form, \eqref{eq_tangential} gives $(\nabla_X\Ric)(Y,Z)=-f'\inner{(\nabla^{\Sigma}_X\cS)Y}{Z}=0$.

Our main result shows that, under constant scalar curvature, the geometry of a single regular level hypersurface already forces rigidity.

\begin{theorem}\label{thm_rigid}
Let $(M^n,g,f,\lambda)$ be a complete gradient Ricci soliton with $\lambda\neq0$ and constant scalar curvature. If some regular level hypersurface $\Sigma$ of $f$ satisfies
\[
	\nabla^{\Sigma}\cS=0\qquad\text{and}\qquad\nabla^{\Sigma}\Ric_{\Sigma}=0,
\]
then $M$ is rigid. In particular, if $\lambda>0$, then $M$ is isometric to a finite quotient of $N^k\times\R^{n-k}$, where $N$ is a compact Einstein manifold with $\Ric_N=\lambda g_N$, and $R=k\lambda$.
\end{theorem}

Conversely, every regular level hypersurface of a finite quotient of $N^k\times\R^{n-k}$ as above is locally isometric to a Riemannian product $N\times\mathbb{S}^{n-k-1}(r)$, and hence satisfies both conditions. If $\lambda>0$ and $R=n\lambda$, then $\Delta f=n\lambda-R=0$, and $|\nabla f|^2-2\lambda f$ is constant, since $R+|\nabla f|^2-2\lambda f$ is constant \cite{ccz,pewy}, and the Bochner formula gives $\nabla^2f=0$, so that $M$ is a compact Einstein manifold. Hence the conjecture of Cao is equivalent to the existence, on every complete gradient shrinking Ricci soliton with constant scalar curvature $R<n\lambda$, of a regular level hypersurface $\Sigma$ with $\nabla^{\Sigma}\cS=0$ and $\nabla^{\Sigma}\Ric_{\Sigma}=0$.

By Theorems \ref{thm_mwp} and \ref{thm_one_level}, some regular level hypersurface satisfies both conditions if and only if every regular level hypersurface of $f$ has parallel second fundamental form. Thus Theorem \ref{thm_rigid} is equivalent to the following result.

\begin{corollary}\label{cor_open}
Let $(M^n,g,f,\lambda)$ be a complete gradient Ricci soliton with $\lambda\neq0$ and constant scalar curvature. If every regular level hypersurface of $f$ has parallel second fundamental form, then $M$ is rigid. This holds, in particular, if every regular point of $f$ has a neighborhood isometric to a multiply warped product \eqref{eq_mwp} on which $f$ depends only on $s$.
\end{corollary}

In particular, the conjecture of Cao holds for complete gradient shrinking Ricci solitons whose regular level hypersurfaces have parallel second fundamental form, a large class which contains, for instance, all the multiply warped products \eqref{eq_mwp} with potential function depending only on the base. When $k\leq2$, the Ricci tensor of such a soliton has at most three distinct eigenvalues, and rigidity also follows from \cite[Theorem 9]{fr16}. In terms of Definition \ref{def_tangential}, Corollary \ref{cor_open} can be stated as follows.

\begin{theorem}\label{thm_cao}
Let $(M^n,g,f,\lambda)$ be a complete gradient shrinking Ricci soliton with constant scalar curvature and tangentially parallel Ricci tensor. Then $M$ is isometric to a finite quotient of $N^k\times\R^{n-k}$, where $N$ is a compact Einstein manifold.
\end{theorem}

By the discussion following Definition \ref{def_tangential}, the complete shrinking solitons of \cite[Theorem B]{angenent}, which are not rigid and have nonconstant scalar curvature, have tangentially parallel Ricci tensor, so the assumption on the scalar curvature in Theorems \ref{thm_rigid} and \ref{thm_cao} cannot be removed. In the shrinking case, however, it is not needed when the level hypersurfaces have at most one nonzero principal curvature, as is the case for the regular level hypersurfaces $N\times\mathbb{S}^{n-k-1}(r)$ of the rigid examples.

\begin{theorem}\label{thm_one_curvature}
Let $(M^n,g,f,\lambda)$ be a complete gradient shrinking Ricci soliton. Suppose that the regular level hypersurfaces of $f$ have parallel second fundamental form and at most one nonzero principal curvature. Then $M$ is rigid.
\end{theorem}

This contains the case of totally umbilical level hypersurfaces. In this case, when $n\geq4$, the scalar curvature is constant on each level hypersurface, by Proposition \ref{prop_level}, so the tensor $D$ of Cao and Chen vanishes \cite[Proposition 3.1]{caoqian}, and the result follows from \cite[Corollary 5.1]{caoqian}. Compare with the classification of complete rotationally invariant shrinking solitons by Kotschwar \cite{kot08}.

The proof of Corollary \ref{cor_open} uses the critical set $F$ of $f$. After normalizing $f$, we have $f=\lambda r^2/2$, where $r$ is the distance to $F$, and, if $\cS$ is the shape operator with respect to $\nabla r$ and $H=\trg\cS$ is the mean curvature, then $rH=r^2|\cS|^2$ is constant. Smoothness of $g$ at $F$ shows that at most one fiber collapses at $F$, and a monotone quantity shows that its warping function is $r$ and that the other ones are constant. Hence $\nabla_{\nabla r}\cS+\cS^2=0$, which gives radial flatness. Theorem \ref{thm_rigid} then follows from Theorem \ref{thm_one_level}. For Theorem \ref{thm_one_curvature}, we show that the universal covering of $M$ splits off an Einstein factor, and that the other factor has totally umbilical level hypersurfaces, so that a result of Cao and Chen \cite{caoqian} (or, in dimensions two and three, the known classifications) applies.

\section{Preliminaries}\label{sec_prelim}

All manifolds are assumed to be connected. We adopt the following convention for the curvature tensor:
\[
	\Rm(X,Y,Z,T)=\inner{\nabla_{Y}\nabla_{X}Z-\nabla_{X}\nabla_{Y}Z+\nabla_{[X,Y]}Z}{T},
\]
so that $\Rm(X,Y,X,Y)$ is the sectional curvature of $\mathrm{span}\{X,Y\}$ for orthonormal $X,Y$, and $\Ric(Y,Z)=\sum_i\Rm(e_i,Y,e_i,Z)$ for an orthonormal frame $\{e_i\}$. We use the same symbol for the Ricci tensor and the corresponding self-adjoint endomorphism. Recall that a gradient Ricci soliton satisfies
\begin{align}\label{eq_identities}
	R+\Delta f=n\lambda,\qquad \nabla R=2\Ric(\nabla f,\cdot),\qquad R+|\nabla f|^2-2\lambda f=\mathrm{const},
\end{align}
see \cite{ccz,pewy}, and $\Delta_fR=2\lambda R-2|\Ric|^2$, where $\Delta_f=\Delta-\inner{\nabla f}{\nabla\cdot}$. Differentiating \eqref{fundeq} also gives
\begin{align}\label{eq_soliton_commutation}
	\Rm(X,Y,\nabla f,Z)=(\nabla_X\Ric)(Y,Z)-(\nabla_Y\Ric)(X,Z).
\end{align}

Let $\Sigma$ be a hypersurface with unit normal field $E_1$, shape operator $\cS X=\nabla_XE_1$ and mean curvature $H=\trg\cS$, and denote its intrinsic quantities as in Section \ref{sec_intro}. We recall the Codazzi, Gauss and Riccati equations for families of equidistant hypersurfaces, in the form used for cohomogeneity one metrics by Eschenburg and Wang \cite{ew00}; see also \cite[(1.2)--(1.4)]{dw11} and \cite[Section 2.2]{tranyamazaki}. With our conventions,
\begin{align}\label{eq_codazzi}
	\Ric(E_1,X)=(\divg_{\Sigma}\cS)(X)-X(H),\qquad X\in T\Sigma,
\end{align}
where $(\divg_{\Sigma}\cS)(X)=\sum_a\inner{(\nabla^{\Sigma}_{e_a}\cS)e_a}{X}$ for an orthonormal frame $\{e_a\}$ of $T\Sigma$, and, for $Y,T\in T\Sigma$,
\begin{align}\label{eq_gauss}
	\Ric(Y,T)=\Rm(E_1,Y,E_1,T)+\Ric_{\Sigma}(Y,T)-H\inner{\cS Y}{T}+\inner{\cS^2Y}{T}.
\end{align}

Suppose now that $s$ is a smooth function with $|\nabla s|=1$, and set $E_1=\nabla s$. Then $\nabla_{E_1}E_1=0$, and following the integral curves of $E_1$ from a level hypersurface of $s$ we obtain, near any point, coordinates $(s,x^2,\ldots,x^n)$ in which
\[
	g=\mathrm{d}s^2+\sum_{i,j\geq2}g_{ij}(s,x^2,\ldots,x^n)\,\mathrm{d}x^i\otimes\mathrm{d}x^j.
\]
We write $g_s$ for the metric induced on the level hypersurfaces of $s$, regarded as a family of metrics in the coordinates $x^2,\ldots,x^n$, so that $\nabla^{\Sigma}$ and $\Ric_{\Sigma}$ are the Levi-Civita connection and the Ricci operator of $g_s$. We denote by a prime the derivative with respect to $s$ of functions and of the components of tensors in these coordinates. Then $g_s'=2g_s(\cS\cdot,\cdot)$, see \cite[Section 1]{dw11} and \cite[Section 2.2]{tranyamazaki},
and $\cS'=\nabla_{E_1}\cS$ on tangent vectors, since $\partial_s$ commutes with the coordinate fields. The Riccati equation $\Rm(E_1,Y,E_1,T)=-\inner{(\cS'+\cS^2)Y}{T}$ and \eqref{eq_gauss} give
\begin{align}\label{eq_riccati}
	\Ric(Y,T)=\Ric_{\Sigma}(Y,T)-H\inner{\cS Y}{T}-\inner{\cS'Y}{T},\qquad \Ric(E_1,E_1)=-H'-|\cS|^2,
\end{align}
for $Y,T$ tangent to the level hypersurfaces of $s$, where $|\cS|^2=\trg\cS^2$. Together with \eqref{eq_codazzi}, this shows that the Ricci tensor of $g$ is determined by the intrinsic Ricci tensor of the level hypersurfaces of $s$, their shape operator and its evolution along $E_1$, as observed by Tran and Yamazaki \cite[Section 2.2]{tranyamazaki}.

We finally recall the connection and the Ricci tensor of a multiply warped product with a one-dimensional base (see \cite{dobarro} and \cite[Lemma 2.5]{bhs2026}).

\begin{lemma}\label{lem_mwp_ricci}
Let $M=I\times_{h_1}N_1^{r_1}\times\cdots\times_{h_k}N_k^{r_k}$ be a multiply warped product as in \eqref{eq_mwp}, and let $V$ and $W$ be lifts of vector fields on $N_i$ and $N_j$, respectively. Then $\nabla_{\partial_s}\partial_s=0$, $\nabla_V\partial_s=\nabla_{\partial_s}V=\frac{h_i'}{h_i}V$, $\Ric(\partial_s,V)=0$, $\Ric(V,W)=0$ if $i\neq j$,
\[
	\Ric(\partial_s,\partial_s)=-\sum_{\ell=1}^kr_\ell\frac{h_\ell''}{h_\ell},
\]
and, if $i=j$,
\[
	\Ric(V,W)=\Ric_{N_i}(V,W)-\left(\frac{h_i''}{h_i}+(r_i-1)\frac{(h_i')^2}{h_i^2}+\sum_{\ell\neq i}r_\ell\frac{h_i'h_\ell'}{h_ih_\ell}\right)g(V,W).
\]
\end{lemma}

\section{Level hypersurfaces with parallel second fundamental form}\label{sec_mwp}

\begin{proof}[Proof of Proposition \ref{prop_level}]
We work on one connected component of $\Sigma$. Since $\nabla^{\Sigma}\cS=0$, $H$ is constant on $\Sigma$ and $\divg_{\Sigma}\cS=0$. By \eqref{eq_codazzi}, $\Ric(E_1,X)=0$ for every $X\in T\Sigma$, so $E_1$ is an eigenvector of $\Ric$ along $\Sigma$. Using \eqref{fundeq} and \eqref{eq_identities}, for $X\in T\Sigma$ we obtain
\[
	X\left(|\nabla f|^2\right)=2\nabla^2f(\nabla f,X)=-2\Ric(\nabla f,X)=0,\qquad X(R)=2\Ric(\nabla f,X)=0.
\]
Hence $|\nabla f|$ and $R$ are constant on $\Sigma$. Moreover, $\nabla^2f(X,Y)=|\nabla f|\inner{\cS X}{Y}$ for $X,Y\in T\Sigma$, so that $\Ric_T=\lambda\,\Id-|\nabla f|\,\cS$. The eigenvalues of $\cS$ are constant on $\Sigma$, because $\cS$ is parallel, and $\Ric(E_1,E_1)=R-\trg\Ric_T$ is constant as well.
\end{proof}

For the remainder of this section, every regular level hypersurface of $f$ has parallel second fundamental form. By Proposition \ref{prop_level}, around each regular point there is a function $s$ with $\mathrm{d}s=\mathrm{d}f/|\nabla f|$, so that $|\nabla s|=1$, $E_1=\nabla s$, $f=f(s)$ and $f'=|\nabla f|$, and the discussion of Section \ref{sec_prelim} applies. The following identity is \cite[(1.16)]{dw11} in our notation, see also \cite[(2.6)]{tranyamazaki}.

\begin{lemma}\label{lem_Sprime}
Let $(M^n,g,f,\lambda)$ be a gradient Ricci soliton, and let $s$ be a smooth function with $|\nabla s|=1$ such that $f$ depends only on $s$. In the coordinates of Section \ref{sec_prelim}, with $E_1=\nabla s$,
\begin{align}\label{eq_Sprime}
	\cS'=\Ric_{\Sigma}-\lambda\,\Id+(f'-H)\,\cS,\qquad \Ric(E_1,E_1)=\lambda-f''.
\end{align}
\end{lemma}

\begin{proof}
Since $\nabla^2f=f''\,\mathrm{d}s\otimes\mathrm{d}s+f'\,\nabla^2s$ and $\nabla^2s(Y,T)=\inner{\cS Y}{T}$ for tangent $Y,T$, the soliton equation gives $\Ric(E_1,E_1)=\lambda-f''$ and $\Ric_T=\lambda\,\Id-f'\cS$. Comparing with \eqref{eq_riccati}, we obtain \eqref{eq_Sprime}.
\end{proof}

\begin{proof}[Proof of Theorem \ref{thm_mwp}]
Fix a regular point and coordinates $(s,x^2,\ldots,x^n)$ around it as in Section \ref{sec_prelim}, defined on a product $I\times V$, with $g=\mathrm{d}s^2+g_s$. Let $\Gamma_{ij}^{\ell}(s,x)$ be the Christoffel symbols of $g_s$, and write $\cS_{ij}=g_s(\cS\partial_i,\partial_j)$. Since $g_{ij}'=2\cS_{ij}$, the variation formula for the Christoffel symbols gives
\[
	\left(\Gamma_{ij}^{\ell}\right)'=\sum_m g^{\ell m}\Big((\nabla^{\Sigma}_{\partial_i}\cS)_{jm}+(\nabla^{\Sigma}_{\partial_j}\cS)_{im}-(\nabla^{\Sigma}_{\partial_m}\cS)_{ij}\Big)=0.
\]
Hence all the metrics $g_s$ have the same Levi-Civita connection, which we still denote by $\nabla^{\Sigma}$, and they are all parallel with respect to it.

Shrinking $V$, we may assume that it is simply connected. Let $\mathrm{Hol}$ be the holonomy group of $(V,g_{s_0})$ at a point $x_0$, and decompose $T_{x_0}V=W_0\oplus W_1\oplus\cdots\oplus W_q$ orthogonally, where $W_0$ is the subspace fixed by $\mathrm{Hol}$ and $W_1,\ldots,W_q$ are irreducible and nontrivial. By the local de Rham theorem, after shrinking $V$ this gives a product decomposition of $(V,g_{s_0})$ into a flat factor and factors $V_1,\ldots,V_q$, and $\mathrm{Hol}=\mathrm{Hol}_1\times\cdots\times\mathrm{Hol}_q$, where $\mathrm{Hol}_j$ is the holonomy group of $V_j$, which acts only on $W_j$. A $\nabla^{\Sigma}$-parallel symmetric bilinear form corresponds to a self-adjoint endomorphism $B$ of $T_{x_0}V$ commuting with $\mathrm{Hol}$. For $j\geq1$ and any other summand $W$, the component $W_j\to W$ of $B$ is $\mathrm{Hol}_j$-equivariant, and $\mathrm{Hol}_j$ acts trivially on $W$. It therefore vanishes on the vectors $hw-w$, $h\in\mathrm{Hol}_j$, $w\in W_j$, which span $W_j$ because $\mathrm{Hol}_j$ fixes no nonzero vector of $W_j$. Thus $B$ preserves $W_j$, and, since $B$ is self-adjoint, its eigenspaces in $W_j$ are $\mathrm{Hol}_j$-invariant, so, by irreducibility, $B$ is a multiple of the identity on $W_j$. Consequently, on each $V_j$, $g_s$ is obtained from $g_{s_0}$ by multiplication by a positive function of $s$ alone, and $g_s$ has no mixed terms involving $V_j$.

On the flat factor, the intrinsic Ricci operator vanishes, and, by Proposition \ref{prop_level}, $f'-H$ is constant on each level hypersurface. Hence, by \eqref{eq_Sprime}, the restriction $\cS_0$ of the shape operator to the flat factor satisfies $\cS_0(s)=a(s)\cS_0(s_0)+b(s)\,\Id$, where
\[
	a(s)=\exp\left(\int_{s_0}^s(f'-H)\right),\qquad
	b(s)=-\lambda a(s)\int_{s_0}^s a(t)^{-1}\,\mathrm{d}t.
\]
The eigenspaces of $\cS_0(s_0)$ are parallel distributions. On each one, $g_s'=2\xi_i(s)g_s$, and hence $g_s=h_i(s)^2g_{s_0}$, where $h_i=\exp(\int_{s_0}^s\xi_i)$. Together with the nonflat factors, this gives \eqref{eq_mwp} on $I\times V$, with $g_{N_i}$ the restriction of $g_{s_0}$ to the corresponding factor.

On every factor, $\cS=\xi_i\,\Id$ with $\xi_i=h_i'/h_i$. Since the metrics $g_s$ have the same Levi-Civita connection, they have the same Ricci tensor, regarded as a $(0,2)$-tensor. Equation \eqref{eq_Sprime} then gives
\begin{align}\label{eq_fiber_einstein}
	\Ric_{N_i}=h_i^2\left(\xi_i'+\lambda-(f'-H)\xi_i\right)g_{N_i}.
\end{align}
The coefficient depends only on $s$, while both tensors on the fiber are independent of $s$. Thus it is a constant $c_i$, and every fiber of dimension at least two is Einstein. It also follows that the intrinsic Ricci tensor of every regular level hypersurface is parallel. Proportional warping functions may be grouped into a single Einstein fiber. To see that the resulting principal curvatures are distinct at a generic value of $s$, put $\kappa_i=c_i/h_i^2$, so that $\xi_i'=\kappa_i-\lambda+(f'-H)\xi_i$ and $\kappa_i'=-2\xi_i\kappa_i$. For two factors, $\delta=\xi_i-\xi_j$ and $\varepsilon=\kappa_i-\kappa_j$ satisfy the homogeneous linear system
\[
	\delta'=\varepsilon+(f'-H)\delta,\qquad \varepsilon'=-2\kappa_i\delta-2\xi_j\varepsilon.
\]
If $(\delta,\varepsilon)$ vanishes at one value of $s$, it vanishes identically, and $h_i$ and $h_j$ are proportional. Otherwise, at a zero of $\delta$ we have $\delta'=\varepsilon\neq0$, so the zeros of $\delta$ are simple and hence isolated.
\end{proof}

\section{A single level hypersurface}\label{sec_one}

We first describe the tangential derivatives of the Ricci tensor when $\nabla f$ is an eigenvector of $\Ric$.

\begin{lemma}\label{lem_ricci_form}
Let $(M^n,g,f,\lambda)$ be a gradient Ricci soliton such that $\Ric(\nabla f,X)=0$ for every $X$ orthogonal to $\nabla f$, and let $\Sigma$ be a regular level hypersurface of $f$. Then, for $X,Y,Z\in T\Sigma$,
\[
	(\nabla_X\Ric)(Y,Z)=-|\nabla f|\inner{(\nabla^{\Sigma}_X\cS)Y}{Z}.
\]
In particular, the Ricci tensor is tangentially parallel if and only if every regular level hypersurface of $f$ has parallel second fundamental form.
\end{lemma}

\begin{proof}
Since $\Ric(E_1,\cdot)$ vanishes on $T\Sigma$, \eqref{eq_tangential} gives
\[
	(\nabla_X\Ric)(Y,Z)=(\nabla^{\Sigma}_X\Ric_T)(Y,Z).
\]
As in the proof of Proposition \ref{prop_level}, we have $X(|\nabla f|^2)=-2\Ric(\nabla f,X)=0$ and $\Ric_T=\lambda\,\Id-|\nabla f|\,\cS$.
\end{proof}

Let now $s$ be a smooth function with $|\nabla s|=1$, and use the coordinates of Section \ref{sec_prelim}, with $E_1=\nabla s$. We denote by $\nabla^{\Sigma}\nabla^{\Sigma}f$ the Hessian, with respect to $g_s$, of the restriction of $f$ to the level hypersurfaces of $s$, regarded as a self-adjoint endomorphism. For $Y,T$ tangent to these hypersurfaces, $\nabla_YT=\nabla^{\Sigma}_YT-\inner{\cS Y}{T}E_1$, so that $\nabla^2f(Y,T)=(\nabla^{\Sigma}\nabla^{\Sigma}f)(Y,T)+f'\inner{\cS Y}{T}$, and $\nabla^2f(E_1,E_1)=f''$, since $\nabla_{E_1}E_1=0$. Comparing \eqref{fundeq} with \eqref{eq_riccati}, as in the proof of Lemma \ref{lem_Sprime}, we obtain
\begin{align}\label{eq_cauchy_system}
\begin{split}
	&g_s'=2g_s(\cS\cdot,\cdot),\qquad f''=\lambda+H'+|\cS|^2,\\
	&\cS'=\Ric_{g_s}-\lambda\,\Id+(f'-H)\,\cS+\nabla^{\Sigma}\nabla^{\Sigma}f.
\end{split}
\end{align}

\begin{proof}[Proof of Theorem \ref{thm_one_level}]
By \cite[Corollary 1.3 and the remark following it]{kotblms}, $(M,g)$ is a real analytic manifold and $\nabla f$ is analytic, hence so is $f$ (see also \cite{bando,ivey96}). Let $c=f(\Sigma)$, and let $V$ be a small connected open subset of $\Sigma$, with induced metric $g_0$ and shape operator $\cS_0$. By Proposition \ref{prop_level}, $|\nabla f|$ is equal to a constant $\beta>0$ on $V$. We first show that the level hypersurfaces of $f$ have parallel second fundamental form near $V$.

Since $\cS_0$ is parallel, its eigendistributions are parallel and give a local product decomposition of $(V,g_0)$. The Ricci tensor of a product is block diagonal, so $\Ric_{g_0}$ commutes with $\cS_0$. As $\Ric_{g_0}$ is also parallel, the common eigendistributions of $\cS_0$ and $\Ric_{g_0}$ are parallel, and, after shrinking $V$,
\begin{align}\label{eq_cauchy_product}
	(V,g_0)=(V_1,g_1)\times\cdots\times(V_q,g_q),\qquad \Ric_{g_i}=\kappa_ig_i,\qquad \cS_0=\alpha_i\,\Id\ \text{ on }TV_i,
\end{align}
where $\alpha_i$ and $\kappa_i$ are constants. Let $r_i=\dim V_i$. These factors are analytic, since $g_0$ and the parallel distributions are analytic. Following the geodesics orthogonal to $V$ with initial velocity $\nabla f/|\nabla f|$, we obtain analytic coordinates $(s,x)$ near $V$, with $V=\{s=0\}$, in which $g=\mathrm{d}s^2+g_s$ as in Section \ref{sec_prelim}. Then $(g_s,\cS,f)$ satisfies \eqref{eq_cauchy_system}.

Let $h_i>0$, $\xi_i$ and $\phi$ be the analytic solutions, near $s=0$, of
\[
	h_i'=\xi_ih_i,\qquad
	\xi_i'=\frac{\kappa_i}{h_i^2}-\lambda+\left(\phi'-\sum_jr_j\xi_j\right)\xi_i,\qquad
	\phi''=\lambda+\sum_jr_j\left(\xi_j'+\xi_j^2\right),
\]
with $h_i(0)=1$, $\xi_i(0)=\alpha_i$, $\phi(0)=c$ and $\phi'(0)=\beta$. After the equations for $\xi_j'$ are substituted in the last one, this is a system of ordinary differential equations in normal form. Put $\widehat{g}_s=\sum_ih_i(s)^2g_i$, $\widehat{\cS}=\xi_i\,\Id$ on $TV_i$ and $\widehat{f}=\phi(s)$. Since $\Ric_{\widehat{g}_s}=(\kappa_i/h_i^2)\,\Id$ on $TV_i$ and $\widehat{f}$ is constant on the level hypersurfaces of $s$, the triple $(\widehat{g}_s,\widehat{\cS},\widehat{f})$ also satisfies \eqref{eq_cauchy_system}. By \eqref{eq_cauchy_product}, and since $f=c$ and $f'=\inner{\nabla f}{E_1}=\beta$ on $V$, the functions $(g_s,\cS,f,f')$ and $(\widehat{g}_s,\widehat{\cS},\widehat{f},\widehat{f}')$ coincide at $s=0$. Differentiating \eqref{eq_cauchy_system} $j$ times with respect to $s$, we see that the derivatives of order $j+1$ of $(g_s,\cS,f,f')$ with respect to $s$ are determined by the derivatives of order at most $j$ and their derivatives in $x$, since $\Ric_{g_s}$ and $\nabla^{\Sigma}\nabla^{\Sigma}f$ involve only derivatives in $x$, and $H'=\trg\cS'$. If the derivatives of order at most $j$ of both quadruples coincide at $s=0$ as functions of $x\in V$, so do their derivatives in $x$. By induction, all the derivatives with respect to $s$ of both quadruples coincide at $s=0$, as functions of $x$. Hence $(g_s,f)$ and $(\widehat{g}_s,\widehat{f})$ have the same partial derivatives at the points of $\{0\}\times V$, and, by analyticity, $g_s=\widehat{g}_s$ and $f=\phi(s)$ near $\{0\}\times V$. Since $\phi'(0)=\beta>0$, the level hypersurfaces of $f$ near $V$ are the level hypersurfaces of $s$, which are Riemannian products whose shape operator is a multiple of the identity on each factor, so they have parallel second fundamental form. Moreover, by Lemma \ref{lem_mwp_ricci}, $\Ric(\nabla f,X)=0$ near $V$ for every $X$ orthogonal to $\nabla f$.

Now consider the analytic tensor
\[
	A(X,Y)=\Ric(\nabla f,X)\,\mathrm{d}f(Y)-\Ric(\nabla f,Y)\,\mathrm{d}f(X).
\]
It vanishes on a nonempty open set, hence on $M$, since $M$ is connected. If $X$ is orthogonal to $\nabla f$, then $A(X,\nabla f)=|\nabla f|^2\Ric(\nabla f,X)$, so $\Ric(\nabla f,X)=0$ at every regular point, and trivially at critical points. Hence Lemma \ref{lem_ricci_form} applies. Let $L=|\nabla f|^2\,\Id-\nabla f\otimes\mathrm{d}f$ and $\cT(X,Y,Z)=(\nabla_{LX}\Ric)(LY,LZ)$. This tensor is analytic on $M$. On $\cR$, $L=|\nabla f|^2P$, where $P$ is the orthogonal projection onto the tangent spaces of the level hypersurfaces, so, by Lemma \ref{lem_ricci_form}, $\cT$ vanishes at a regular point if and only if $\nabla^{\Sigma}\cS=0$ at this point. By the first part, $\cT$ vanishes on a nonempty open set, hence on $M$. Thus every regular level hypersurface of $f$ has parallel second fundamental form.
\end{proof}

\section{Constant scalar curvature}\label{sec_rigid}

Throughout this section, $(M^n,g,f,\lambda)$ is a complete gradient Ricci soliton with $\lambda\neq0$ and constant scalar curvature $R$. By \eqref{eq_identities}, $\Ric(\nabla f,\cdot)=0$ and, after adding a constant to $f$, we have $|\nabla f|^2=2\lambda f$ (see also \cite{fr16}). Hence $f/\lambda\geq0$, the critical set of $f$ is $F=f^{-1}(0)$, and $\cR=M\setminus F$. We assume that $f$ is nonconstant, since otherwise $M$ is Einstein. The function $r=\sqrt{2f/\lambda}$ satisfies $|\nabla r|=1$ on $\cR$, and $f=\frac{\lambda}{2}r^2$. By \cite[Lemma 6]{fr16}, $F\neq\emptyset$, and, since $f$ is isoparametric, $F$ is a smooth submanifold of $M$ \cite{wang}. The integral curves of $-\nabla r$ are unit speed geodesics along which $r$ decreases with unit speed. By completeness, the one starting at $x\in\cR$ extends to a geodesic $\gamma$ on $[0,r(x)]$ with $r(\gamma(t))=r(x)-t$, so that $\gamma(r(x))\in F$. Hence every point $x\in\cR$ is joined to a point of $F$ by a curve of length $r(x)$ whose interior lies in $\cR$. Since $r$ is $1$-Lipschitz and vanishes on $F$, $r$ is the distance to $F$. We take the shape operator with respect to $\nabla r$ when using $r$ in place of $s$, which is allowed by Lemma \ref{lem_Sprime}.

The next lemma is the standard description of the Hessian of $f$ at its critical set (cf. \cite{wang} and the proof of \cite[Theorem 1]{fr16}).

\begin{lemma}\label{lem_tube}
Let $\gamma$ be a unit speed geodesic with $\gamma(0)=p\in F$ and $r(\gamma(t))=t$ for $0<t<\varepsilon$. In a parallel orthonormal frame of $\gamma'^{\perp}$ along $\gamma$, the matrix $t\,\cS(t)$ of $t$ times the shape operator of the level hypersurface of $r$ through $\gamma(t)$ extends smoothly to $t=0$, and its value at $t=0$ has only the eigenvalues $0$ and $1$.
\end{lemma}

\begin{proof}
For $X,Y$ orthogonal to $\nabla r$, we have $\nabla^2f(X,Y)=\lambda r\inner{\cS X}{Y}$. Thus $t\,\cS(t)$ is the matrix of $\lambda^{-1}\nabla^2f$ restricted to $\gamma'^{\perp}$, which is smooth up to $t=0$. Differentiating $|\nabla f|^2=2\lambda f$ twice at $p$ gives $J^2=\lambda J$ for $J=(\nabla^2f)_p$. Since $\nabla^2f(\gamma',\gamma')=f''=\lambda$, the vector $\gamma'(0)$ is an eigenvector of $J$ with eigenvalue $\lambda$, so $J$ preserves $\gamma'(0)^{\perp}$, where its eigenvalues lie in $\{0,\lambda\}$.
\end{proof}

Since $f=\lambda r^2/2$, equations \eqref{eq_identities} and \eqref{eq_riccati} give
\begin{align}\label{eq_m}
	rH=r^2|\cS|^2=m,\qquad m=n-1-\frac{R}{\lambda}.
\end{align}
If $u_i=r\xi_i$ are the eigenvalues of $r\cS$, with multiplicities $r_i$, then
\begin{align}\label{eq_traces}
	\sum_i r_i u_i=m=\sum_i r_i u_i^2.
\end{align}
These are also $\trg\Ric=R$ and $|\Ric|^2=\lambda R$, since the normal Ricci eigenvalue is zero and the tangential ones are $\lambda(1-u_i)$. The norm identity follows from $\Delta_fR=0$.

\begin{proof}[Proof of Corollary \ref{cor_open}]
If $m=0$, equation \eqref{eq_m} gives $\cS=0$, and radial flatness follows from the Riccati equation. Suppose $m>0$ and fix a regular point $x$. Following the integral curve of $-\nabla r$ through $x$, we reach a point $p\in F$, and reversing it we obtain a geodesic $\gamma$ with $\gamma(0)=p$ and $\gamma(r(x))=x$. Fix $0<\delta<r(x)<T$. By completeness, the integral curves of $\nabla r$ starting at points of $r^{-1}(\delta)$ near $\gamma(\delta)$ are geodesics defined, with $r$ as parameter, for $0<r<T$. Following them from a small neighborhood $V$ of $\gamma(\delta)$ in $r^{-1}(\delta)$, we obtain coordinates as in Section \ref{sec_prelim} on an open set $(0,T)\times V$ containing $\gamma((0,T))$, with $g=\mathrm{d}r^2+g_r$.

The proof of Theorem \ref{thm_mwp} applies on all of $(0,T)\times V$, since it only requires shrinking $V$: the metrics $g_r$, $0<r<T$, have the same Levi-Civita connection, and the decomposition is determined by its holonomy and by the eigenspaces of $\cS_0$ at a single value of $r$. Hence, after shrinking $V$, we have on $(0,T)\times V$ a fixed product decomposition
\[
	g=\mathrm{d}r^2+\sum_i h_i(r)^2g_{N_i},\qquad \Ric_{N_i}=c_i g_{N_i},\qquad \xi_i=\frac{h_i'}{h_i},
\]
with pairwise nonproportional warping functions. By Lemma \ref{lem_mwp_ricci}, $\nabla_{\partial_r}V=\xi_iV$ for $V$ tangent to $N_i$, so the distributions $TN_i$ are invariant under parallel transport along $\gamma$, and we may choose a parallel orthonormal frame along $\gamma$ adapted to them, in which $r\cS$ is diagonal with entries $u_i=r\xi_i$. By Lemma \ref{lem_tube}, each $u_i$ extends smoothly to $r=0$, with $u_i(0)=\varepsilon_i\in\{0,1\}$. Hence $u_i=\varepsilon_i+O(r)$, so that $\xi_i=\varepsilon_i r^{-1}+O(1)$, and integrating $(\log h_i)'=\xi_i$ gives $h_i(r)=a_i r^{\varepsilon_i}(1+O(r))$ with $a_i>0$. We call a factor collapsing if $\varepsilon_i=1$. Equation \eqref{eq_traces} gives $m=\sum_{\varepsilon_i=1}r_i$.

There is only one collapsing factor. Indeed, suppose that two distinct factors $N_i$ and $N_j$ collapse, and let $e_i\in TN_i$ and $e_j\in TN_j$ be unit vectors. Since the level hypersurfaces are Riemannian products, the intrinsic sectional curvature of $\mathrm{span}\{e_i,e_j\}$ is zero, and, since $\cS e_i=\xi_ie_i$ and $\cS e_j=\xi_je_j$, the Gauss equation gives $\Rm(e_i,e_j,e_i,e_j)=-\xi_i\xi_j=-r^{-2}+O(r^{-1})$ along $\gamma$. This is impossible, since $g$ is smooth at $p$ and hence its curvature is bounded near $p$. Denote the collapsing factor by $N_0$ and its warping function by $h$. Since $m=\sum_{\varepsilon_i=1}r_i$, the dimension of $N_0$ is $m$. Rescaling $g_{N_0}$, we arrange $h(r)/r\to1$. Put $u=rh'/h=r\xi_0$, so that $u\to1$. For a unit vector $e$ tangent to $N_0$, we have $\cS e=\xi_0e$ and $\Ric_{\Sigma}(e,e)=\Ric_{N_0}(e,e)=c_0g_{N_0}(e,e)=c_0/h^2$. Hence \eqref{eq_gauss} and $rH=m$ give
\[
	r^2\frac{c_0}{h^2}=r^2\big(\Ric(e,e)-\Rm(\partial_r,e,\partial_r,e)\big)+mu-u^2.
\]
Since the curvature of $g$ is bounded near $p$, the first term on the right-hand side tends to zero as $r\to0^+$, and, since $u\to1$ and $h/r\to1$, we obtain $c_0=m-1$.

We now show that $u=1$, that is, $h=r$. By \eqref{eq_traces}, $m u^2\leq m$, hence $u\leq1$. Since $(\log(h/r))'=(u-1)/r\leq0$ and $h/r\to1$, we have $h\leq r$. Applying \eqref{eq_Sprime} on $TN_0$, with $f'=\lambda r$ and $H=m/r$, gives
\[
	r u'=(m-1)\left(\frac{r^2}{h^2}-u\right)-\lambda r^2(1-u)
	\geq (m-1-\lambda r^2)(1-u).
\]
Thus, by the equality above, the nonnegative function $\Phi(r)=r^{m-1}e^{-\lambda r^2/2}(1-u(r))$ satisfies
\[
	\Phi'(r)=-(m-1)r^{m-2}e^{-\lambda r^2/2}\left(\frac{r^2}{h^2}-1\right)\leq0
\]
and $\lim_{r\to0^+}\Phi(r)=0$, since $u(r)\to1$. Therefore $\Phi=0$ and $u=1$ on $(0,T)$. The norm identity in \eqref{eq_traces} now forces $u_i=0$ on every other factor. Hence $\xi_0=1/r$ and $\xi_i=0$ for $i\neq0$, so $\cS'+\cS^2=0$ on $(0,T)\times V$.

Since $x$ was arbitrary, the Riccati equation gives $\Rm(\nabla f,X,\nabla f,X)=0$ at every regular point, and at a critical point this identity is automatic. Thus $M$ is radially flat and is rigid by \cite[Theorem 1.2]{pewy}. If $\lambda>0$, write $M=N\times_{\Gamma}\R^{n-k}$ with $\Ric_N=\lambda g_N$. By Myers' theorem, $N$ is compact. The properly discontinuous action of $\Gamma$ preserves $N\times\{0\}$, because it acts orthogonally on the Euclidean factor, so $\Gamma$ is finite. Finally, $R=k\lambda$.

For the last assertion, the level hypersurfaces of $f$ in a neighborhood as in the statement are slices of \eqref{eq_mwp}, which are Riemannian products of the manifolds $(N_i,h_i(s)^2g_{N_i})$. By Lemma \ref{lem_mwp_ricci}, their shape operator is $(h_i'/h_i)\,\Id$ on $TN_i$, and, since the projections onto the factors of a Riemannian product are parallel, they have parallel second fundamental form.
\end{proof}

\begin{proof}[Proof of Theorem \ref{thm_cao}]
Since $\Ric(\nabla f,\cdot)=0$, Lemma \ref{lem_ricci_form} shows that every regular level hypersurface of $f$ has parallel second fundamental form, and Corollary \ref{cor_open} applies.
\end{proof}

\begin{proof}[Proof of Theorem \ref{thm_rigid}]
By Theorem \ref{thm_one_level}, every regular level hypersurface of $f$ has parallel second fundamental form, and Corollary \ref{cor_open} applies.
\end{proof}

\begin{remark}
By Lemma \ref{lem_tube}, $m$ is the codimension of $F$ minus one, so that $R=(\dim F)\lambda$. This is the quantization of the scalar curvature obtained by Fern\'andez-L\'opez and Garc\'ia-R\'io in the proof of \cite[Theorem 1]{fr16}.
\end{remark}

\section{Level hypersurfaces with at most one nonzero principal curvature}\label{sec_one_curvature}

In this section, $(M^n,g,f,\lambda)$ is a complete gradient shrinking Ricci soliton as in Theorem \ref{thm_one_curvature}.

\begin{proof}[Proof of Theorem \ref{thm_one_curvature}]
As in the proof of Theorem \ref{thm_one_level}, $(M,g)$ and $f$ are analytic. We assume that $f$ is nonconstant, since otherwise $M$ is Einstein. Fix $x\in\cR$. By Theorem \ref{thm_mwp}, $x$ has a connected and simply connected neighborhood $W\subset\cR$ isometric to a multiply warped product \eqref{eq_mwp}, with pairwise nonproportional warping functions and Einstein fibers, on which $f$ depends only on $s$ and $\cS=\xi_i\,\Id$ on $TN_i$, with $\xi_i=h_i'/h_i$. Since $f'=|\nabla f|$ and $H$ are analytic functions of $s$, the pairs $(h_i,\xi_i)$ are solutions of the analytic system formed by $h_i'=\xi_ih_i$ and \eqref{eq_fiber_einstein}, and hence are analytic. If two of the functions $\xi_i$ were not identically zero, they would both be nonzero on an open interval, which is excluded by our assumption. Moreover, a warping function with $\xi_i=0$ is constant, so, since the warping functions are pairwise nonproportional, $k\leq2$.

Suppose first that $\cS=0$ on $W$. By \eqref{eq_riccati} and \eqref{eq_Sprime}, $\Ric(E_1,E_1)=0$ and $\Ric_T=\lambda\,\Id$, so that $R=(n-1)\lambda$, and the Riccati equation gives $\Rm(E_1,X,E_1,X)=0$ for every $X$. Thus $\mathrm{d}R$ and the analytic tensor $\Rm(\nabla f,\cdot,\nabla f,\cdot)$ vanish on $W$, hence on $M$, and $M$ is rigid by \cite[Theorem 1.2]{pewy}.

Otherwise, after shrinking $W$, we may assume that $\xi_0$ has no zeros, where $N_0$ is the fiber whose warping function $h=h_0$ is not constant. Then $TN_0$ is the eigenspace of the nonzero principal curvature, and we put $m=\dim N_0$. If $k=2$, the other warping function is constant, and we may assume that it is equal to one. By \eqref{eq_fiber_einstein}, the other fiber $N_1$ then satisfies $\Ric_{N_1}=\lambda g_{N_1}$, and $W$ is the Riemannian product of $I\times_hN_0$ and $N_1$. If $k=1$, we let $N_1$ be a point.

Let $\widetilde{M}$ be the universal covering of $M$, with the induced metric $\widetilde{g}$ and potential function $\widetilde{f}$, which is a complete gradient shrinking Ricci soliton, and let $\widetilde{W}$ be a connected component of the preimage of $W$. The distribution tangent to $N_1$ is parallel on $\widetilde{W}$. Since $\widetilde{g}$ is analytic and $\widetilde{M}$ is simply connected, the holonomy group of $\widetilde{g}$ at a point of $\widetilde{W}$ coincides with its local holonomy group \cite[Chapter II, Section 10]{kn}, so this distribution extends to a parallel distribution on $\widetilde{M}$. By the de Rham decomposition theorem \cite[Chapter IV, Section 6]{kn}, $\widetilde{M}$ is isometric to a Riemannian product $P\times Q$, where $Q$ is the integral manifold of this distribution, so that $\widetilde{W}$ is an open subset of $P\times Q$ whose slices $\{y\}\times Q$ contain the factors $N_1$. The Ricci tensor of $P\times Q$ is block diagonal, and hence so is $\nabla^2\widetilde{f}$, by \eqref{fundeq}. Thus $\widetilde{f}(y,z)=f_P(y)+f_Q(z)$, and $(P,g_P,f_P,\lambda)$ and $(Q,g_Q,f_Q,\lambda)$ are complete gradient shrinking Ricci solitons. On $\widetilde{W}$, $\widetilde{f}$ depends only on $s$, so $f_Q$ is constant on an open set, hence on $Q$, by analyticity, and $Q$ is Einstein with $\Ric_Q=\lambda g_Q$.

The soliton $P$ has dimension $m+1$, and we show that it is Einstein or rigid. Suppose first that $m\geq3$. On the open subset of $P$ corresponding to $I\times_hN_0$, the level hypersurfaces of $f_P$ are totally umbilical, and the scalar curvature $R_P=R-\lambda\dim Q$ is constant on them, by Proposition \ref{prop_level}. Hence the tensor $D$ of Cao and Chen of $P$ vanishes on this open set by \cite[Proposition 3.1]{caoqian}. Since $D$ is analytic, it vanishes on $P$, and, by \cite[Corollary 5.1]{caoqian}, $P$ is Einstein, or a finite quotient of the Gaussian soliton $\R^{m+1}$ or of $N^m\times\R$, where $N$ is Einstein. If $m=2$ and $P$ is compact, then $P$ has constant curvature by a result of Ivey \cite{ivey93}. If $m=2$ and $P$ is noncompact, then $P$ is flat or a quotient of $\mathbb{S}^2\times\R$, by the classification of Cao, Chen and Zhu \cite{ccz} (see also \cite{caozhou}). If $m=1$, then $\Ric_P=\frac{1}{2}R_Pg_P$, and \eqref{eq_identities} gives $\nabla R_P=R_P\nabla f_P$, so that $R_P=a\,e^{f_P}$ for a constant $a$. If $P$ is compact, it has constant curvature by a result of Hamilton \cite{ham88}. Otherwise, $R_P\geq0$ and $f_P$ is unbounded above, by \cite[Lemma 2.2 and Theorem 1.1]{caozhou}. Since $R_P\leq 2\lambda f_P+C$ by \eqref{eq_identities}, this forces $a=0$, so that $P$ is flat.

In the remaining cases, $P$ is rigid. If $P$ is Einstein, then $\nabla\Ric_P=0$, and \eqref{eq_soliton_commutation} gives $\Rm_P(X,Y,\nabla f_P,Z)=0$. In all cases, $P$ has constant scalar curvature and is radially flat. Since $\nabla\widetilde{f}=\nabla f_P$ is tangent to $P$, the same holds for $\widetilde{M}$, and hence for $M$. By \cite[Theorem 1.2]{pewy}, $M$ is rigid.
\end{proof}

\section*{Acknowledgments}

This study was financed by the Coordena\c{c}\~ao de Aperfei\c{c}oamento de Pessoal de N\'ivel Superior -- Brasil (CAPES) -- Finance Code 001. The author thanks the Departamento de Matem\'atica of the Universidade de Bras\'ilia, where this work was carried out.

\end{document}